\documentclass[12pt]{amsart}
\usepackage[all]{xy}
\usepackage{pdfsync}
 \usepackage[colorlinks=false]{hyperref}
\usepackage{amssymb} 
\usepackage{amsmath} 
\usepackage{graphicx} 
\usepackage[english]{babel} 
\usepackage{epsfig}
\usepackage{multirow}
\usepackage{color}
\usepackage{enumerate}
\usepackage[T1]{fontenc}
\usepackage{aurical, pbsi,calligra}
\swapnumbers

\theoremstyle{definition}
\newtheorem{definition}{Definition}[section]
\newtheorem{remark}[definition]{Remark}

\theoremstyle{plain}
\newtheorem{lemma}[definition]{Lemma}
\newtheorem{prp}[definition]{Proposition}
\newtheorem{corollary}[definition]{Corollary}
\newtheorem{theorem}[definition]{Theorem}
\usepackage[all]{xy}
\usepackage{pdfsync}
\usepackage{ytableau}

\begin{document}

\title{Polynomial identities of bicommutative algebras, Lie and Jordan elements}
\author{A.S. Dzhumadil'daev, N.A. Ismailov}

\address{Kazakh-British Technical University, Tole bi 59,
Almaty,050000,Kazakhstan}

\address{S.Demirel University,  Toraygirova, 19, Almaty,  050043, Kazakhstan}

\email{dzhuma@hotmail.com, nurlan.ismai@gmail.com}


 \maketitle

\begin{abstract} An algebra with identities $a(bc)=b(ac),$ $(ab)c=(ac)b$ is
called bicommutative. We construct list of identities satisfied by  commutator and anti-commutator products in a free bicommutative algebra. We give criterions for elements of a free bicommutative algebra to be Lie or Jordan.
\end{abstract}

\section{\label{nn}\ Introduction}
One of the common questions in algebra for an adjoint class of a variety via commutator or anti-commutator products is to determine a set of their identities. 
For some varieties of algebras this question is solved or partially solved and for some of them are still open.

Classical example of such approach gives us a variety of associative algebras. Commutator and anti-commutator algebras of an associative algebra satisfy Jacobi and Jordan identities respectively. By Poincar\'e-Birkhoff-Witt (PBW) theorem anti-commutativity and the Jacobi identities gives us complete list of identities for commutator algebras of associative algebras  \cite{Bergman},\cite{Bokut1},\cite{Bokut2}. In other words, any identity satisfied by the commutator product in every associative algebra is a consequence of the anti-commutativity and the Jacobi identities. The case of anti-commutator algebras of associative algebras is completely different. There is no  analogue of PBW theorem. There is no embedding theorem for any Jordan algebra. There exists an  identity of degree eight that is not consequence of commutativity and the Jordan identities \cite{Glennie1},\cite{Glennie2}. For a survey on identities of Jordan algebras see \cite{McCrimmon}. 

Another example concerns  a variety of pre-Lie (right-symmetric) algebras. The commutator algebra of  pre-Lie algebra satisfies anti-commutativity and the Jacobi identities. Moreover, there holds embedding theorem for all Lie algebras \cite{Bokut3},\cite{Dzhum-Lof},\cite{Manchon},\cite{Markl}. In other words, any identity satisfied by the commutator product in every pre-Lie algebra is a consequence of the anti-commutativity and the Jacobi identities. As far as anti-commutator case,  no identity except commutativity holds for any anti-commutator algebra of all pre-Lie algebras \cite{Ber-Lod}.

The first author has studied identities satisfied by the commutator and anti-commutator products in Novikov, Zinbiel, Leibniz and assosymmetric algebras, see \cite{Dzhumadil'daev2},\cite{Dzhumadil'daev3},\cite{Dzhumadil'daev4},\cite{Dzhumadil'daev5}, \cite{Dzhumadil'daev6} and \cite{Dzhumadil'daev7}.          

An algebra with identities 
\begin{equation}\label{f1}a(bc)=b(ac)
\end{equation}
\begin{equation}\label{f2}(ab)c=(ac)b
\end{equation}
is called {\it  bicommutative.} The identity (\ref{f1}) is called {\it left-commutative} and the identity (\ref{f2}) is called {\it right-commutative}. The bicommutative algebras are studied in papers \cite{Burde},\cite{Drensky-Zhakh},\cite{Drensky},\cite{Dzh-Tul},\cite{Dzh-Ism-Tul} and \cite{Kaygorodov-Volkov}. 

In our paper we  study such questions for bicommutative algebras. Let $\mathcal{B}icom$ be a variety of bicommutative algebras. Define $\mathcal{B}icom^{(-)}$ and $\mathcal{B}icom^{(+)}$ as  classes of algebras of a form $A^{(-)}$ and $A^{(+)},$ where $A\in  \mathcal{B}icom.$ 
Recall that {\it commutator} product is defined by  $$[x,y]=xy-yx$$   and {\it anti-commutator} product  by $$\{x,y\}=xy+yx.$$ 
We say that  $A^{(-)}$ and $A^{(+)}$ are  commutator and the anti-commutator algebras of $A\in\mathcal{B}icom$ recpectively. In \cite{Burde} and \cite{Dzh-Tul} there were shown that any algebra of $\mathcal{B}icom^{(-)}$ is a metabelian Lie algebra. We prove that every identity satisfied by the commutator product in all bicommutative algebras is a consequence of anti-commutativity, the Jacobi and the metabelian identities. In the anti-commutator case, we obtain two identities in degree 4, one of them was given in \cite{Dzh-Tul}. We prove that every identity satisfied by the anti-commutator product in all bicommutative algebras is a consequence of commutativity and these two identities. Furthermore, we show that these identities in degree 4 are independent in a free bicommutative algebra.

Let $X$ be a set and $As(X)$  be a free associative algebra generated by $X$. A polynomial in $As(X)$ is called ${\it Lie}$ element of $As(X)$ if it can be expressed by elements of $X$ in terms of commutators. Similarly, a polynomial in $As(X)$ is called ${\it Jordan}$ element of $As(X)$ if it can be expressed by elements of $X$ in terms of anti-commutators. There are two well-known Lie criterions, Dynkin-Specht-Wever criterion \cite{Dynkin},\cite{Specht}, \cite{Wever} and Friedrich's criterion \cite{Friedrichs}. On the other hand,  there is still no a criterion that determines all Jordan elements in $As(X).$  This problem is solved only for some subspaces of the space of all Jordan elements \cite{Cohn},\cite{Robbins}. These problems were considered for other varieties of algebras. For example, Lie criterion for pre-Lie algebras is given in \cite{Markl}, Jordan criterion for Leibniz algebras is given in \cite{Dzhumadil'daev6}.

We consider the analogues of these questions for bicommutative algebras. Let $Bicom(X)$  be a free bicommutative algebra generated by $X$. We give Lie and Joran criterions that enable one to determine whether an element in $Bicom(X)$ is a Lie and a Jordan element, respectively.

\section{\label{nn}\ Statement of results.}
We consider all algebras over a field ${\bf K}$ of characteristic 0, hence any identity is equivalent to a finite set of multilinear identities, see chapter 1 in \cite{ZSSS}. Therefore, we may restrict our attention to multilinear identities. 

\begin{theorem}\label{Lie identities}
Any identity satisfied by the commutator in every bicommutative algebra is a consequence of anti-commutativity, the Jacobi and the metabelian identities:
\begin{equation}\label{anti-com}[a,b]=-[b,a],\end{equation}
\begin{equation}\label{jacobi}[[a,b],c]+[[b,c],a]+[[c,a],b]=0,\end{equation}
\begin{equation}\label{metabelian}[[a,b],[c,d]]=0.\end{equation}
\end{theorem}

\begin{theorem}\label{Jordan identities}
Any identity satisfied by the anti-commutator in every bicommutative algebra is a consequence of commutativity, the minus-Tortken and weak right-commutativity identities:
\begin{equation}\label{com}\{a,b\}=\{b,a\},\end{equation}
\begin{equation}\label{Tortken-minus}\{\{a,b\},\{c,d\}\}-\{\{a,d\},\{c,b\}\}=-\{(a,b,c),d\}+\{(a,d,c),b\}\end{equation} 
\begin{equation}\label{weakrcom}\{\{\{a,b\},c\},d\}=\{\{\{a,b\},d\},c\}\end{equation}
where $(a,b,c)=\{a,\{b,c\}\}-\{\{a,b\},c\}$ is the associator of $a,b,c.$ 
\end{theorem}
\begin{remark}
Any Novikov algebra under anti-commutator satisfies an identity so-called Tortken, see \cite{Dzhumadil'daev2}
\begin{equation}\label{Tortken}
\{\{a,b\},\{c,d\}\}-\{\{a,d\},\{c,b\}\}=\{(a,b,c),d\}-\{(a,d,c),b\}.\end{equation}  
We note that (\ref{Tortken-minus}) differs from (\ref{Tortken}) in signs on the right hand-side and propose to call (\ref{Tortken-minus}) by the {\it minus-Tortken} identity. The identity (\ref{weakrcom}) is a right-commutativity identity of three elements if the first element is decomposable. For this reason, we call it by {\it weak right-commutativity} identity. \end{remark} Suppose that $X$ is an ordered set. Recall a basis of free bicommutative algebra generated by $X$ in \cite{Dzh-Ism-Tul}. In order to obtain base elements in degree $n$, we consider Young diagrams of form $(n)$ and $(n-k,1^k)$ where $k=1,\ldots,n-2.$ We fill the Young diagrams by elements of $X$ so that $x_1\leq x_2\leq\ldots\leq x_k,$ $y_1\leq y_2\leq\ldots\leq y_l$ and $k,l>0$ for $x_1,\ldots,x_k,y_1,\ldots,y_l\in X$ and then correspond them to monomials of bicommutative base elements as follows  

$$\ytableausetup
{mathmode, boxsize=1.2 em}
\begin{ytableau}
x_1 & y_1& y_2 & \dots
&  y_l \\
x_2 \\
\vdots \\
x_k\\
\end{ytableau}\longmapsto x_k(\cdots(x_2((\cdots((x_1y_1)y_2)\cdots)y_l))\cdots).$$

As an example, we give the construction of multilinear base elements of $Bicom(\{x,y,z\})$ and assume that $x<y<z.$
\begin{center}
$\ytableausetup
{mathmode, boxsize=1em}
\begin{ytableau}
x & y & z\\
\end{ytableau}\mapsto (xy)z,         \ytableausetup{mathmode, boxsize=1em}
\begin{ytableau}
y & x & z\\
\end{ytableau}\mapsto (yx)z,\ytableausetup{mathmode, boxsize=1em}
\begin{ytableau}
z & x & y\\
\end{ytableau}\mapsto (zx)y,$
\end{center}

\bigskip

\begin{center}
$\ytableausetup
{mathmode, boxsize=1em}
\begin{ytableau}
x & z\\
y\\
\end{ytableau}\mapsto y(xz),\quad \ytableausetup
{mathmode, boxsize=1em}
\begin{ytableau}
x & y\\
z\\
\end{ytableau}\mapsto z(xy),\quad \ytableausetup
{mathmode, boxsize=1em}
\begin{ytableau}
y & x\\
z\\
\end{ytableau}\mapsto z(yx).$
\end{center}
\bigskip

To formulate Jordan and Lie criterions, we introduce two linear maps, involution and Dynkin maps, on $Bicom(X)$.

Define a conjugation map $\ast:Bicom(X)\rightarrow Bicom(X)$ on base elements by
$$x\longmapsto x$$
and
\begin{center}
$\ytableausetup
{mathmode, boxsize=1.2 em}
\begin{ytableau}
x_1 & y_1& y_2 & \dots
&  y_l \\
x_2 \\
\vdots \\
x_k\\
\end{ytableau}\longmapsto\ytableausetup
{mathmode, boxsize=1.2 em}
\begin{ytableau}
y_1 & x_1& x_2 & \dots
&  x_k \\
y_2 \\
\vdots \\
y_l\\
\end{ytableau}.$
\end{center}
\begin{prp} 
$(a^{\ast})^{\ast}=a$ and
$(ab)^{\ast}=b^{\ast}a^{\ast}$ for any $a,b\in Bicom(X).$ In other words, $\ast$ is an involution map on $Bicom(X).$

Moreover, if $a$ and $b$ are homogenous and homogenous degrees are more than one, then $(ab)^{\ast}=a^{\ast}b^{\ast}.$
\end{prp}
\begin{proof} Follows from Lemma 3.3 of \cite{Dzh-Ism-Tul}.
\end{proof}
An element $f\in Bicom(X)$ is called {\it symmetric} and {\it skew-symmetric} if $f^{\ast}=f$ and $f^{\ast}=-f,$ respectively.
\begin{theorem}\label{Jordan element}
Let $f$ be an element of $Bicom(X).$ Then $f$ is Jordan if and only if $f$ is symmetric.  
\end{theorem}

The {\it Dynkin map} $D:Bicom(X)\rightarrow Bicom(X)$
is a linear map, defined on base elements as follows
$$\ytableausetup
{mathmode, boxsize=1.2 em}
\begin{ytableau}
a_1 & b_1& b_2 & \dots
&  b_l \\
a_2 \\
\vdots \\
a_k\\
\end{ytableau}\longmapsto\frac{1}{2}[a_k, [\cdots[a_2,[[\cdots[[a_1, b_1], b_2]\cdots], b_l]]\cdots]].$$
where $k+l=n.$
We define {\it head} of $f\in Bicom(X)$, $head(f),$ as a projection to a linear space, generated by base elements, where $x_1$ appears in the second column of Young diagram $(n)$ and appears in the first column of Young diagram $(2,1^{n-2})$, that is, 

\begin{center}
$\ytableausetup 
{mathmode, boxsize=1.2 em}
\begin{ytableau}
$ $ & x_1 & $ $ & \dots
&  $ $ 
\end{ytableau}
$ \quad \quad and \quad  \quad  $\ytableausetup
{mathmode, boxsize=1.2 em}
\begin{ytableau}
x_1 & $ $ \\
$ $ \\
\vdots \\
$ $\\
\end{ytableau}.$
 \end{center}
The {\it tail} of $f$ is the remaining part of $f$ from $head(f)$, that is, $tail(f)=f-head(f).$ 
As example, if

$f=((x_1x_2)x_3)x_4+((x_3x_1)x_2)x_4-x_2((x_1x_3)x_4)-x_4((x_3x_1)x_2)-$$$x_4(x_3(x_2x_1))-x_4(x_2(x_1x_3))=$$

$\ytableausetup
{mathmode, boxsize=1.2 em}
\begin{ytableau}
x_1 & x_2& x_3
&  x_4 
\end{ytableau}+\ytableausetup
{mathmode, boxsize=1.2 em}
\begin{ytableau}
x_3 & x_1& x_2
&  x_4 
\end{ytableau}-\begin{ytableau}
x_1 & x_3& x_4  \\
x_2\\
\end{ytableau}-\begin{ytableau}
x_3 & x_1& x_2  \\
x_4\\
\end{ytableau}-\begin{ytableau}
x_2 & x_1 \\
x_3\\
x_4\\
\end{ytableau}-\begin{ytableau}
x_1 & x_3 \\
x_2\\
x_4\\
\end{ytableau}$

then 

$$head(f)=\ytableausetup
{mathmode, boxsize=1.2 em}
\begin{ytableau}
x_3 & x_1& x_2
&  x_4 
\end{ytableau}-\begin{ytableau}
x_1 & x_3 \\
x_2\\
x_4\\
\end{ytableau}$$

and 

$$tail(f)=\ytableausetup
{mathmode, boxsize=1.2 em}
\begin{ytableau}
x_1 & x_2& x_3
&  x_4 
\end{ytableau}-\begin{ytableau}
x_1 & x_3& x_4  \\
x_2\\
\end{ytableau}-\begin{ytableau}
x_3 & x_1& x_2  \\
x_4\\
\end{ytableau}-\begin{ytableau}
x_2 & x_1 \\
x_3\\
x_4\\
\end{ytableau}.$$

\begin{theorem}\label{Lie elements}
Let $f\in Bicom(X).$ Then $f$ is a Lie element if and only if $D(head(f))=f.$
\end{theorem}

By $L(X)$ and $J(X)$ we denote the free commutator and free anti-commutator algebras, respectively  generated on $X$ of $Bicom(X).$ Let $L_n$ and $J_n$ be their homogenous part  of degree $n$, respectively. An element of $Bicom(X)$ is called Lie if it belongs to $L(X).$ Similarly an element of $Bicom(X)$ is called Jordan if it belongs to $J(X).$ 

\begin{corollary}\label{relation}
$L_n\cap J_n=\emptyset$ if $n$ is even,
$L_n\subseteq J_n$ if $n$ is odd.\end{corollary}
A subspace of symmetric elements is generated by elements of a form $e^{(+)}=e+e^\ast,$ similarly a subspace of skew-symmetric elements is generated by elements of a form $e^{(+)}=e+e^\ast$, where $e$ is the base element of $Bicom(X).$

\section{\label{nn}\ Proof of Theorems \ref{Jordan identities} and \ref{Jordan element}}
\begin{lemma}\label{Filtr-plus}
$J_{n+2}=\{J_{n+1},J_1\}+\{J_{n},J_2\}$ for $n>0.$ 
\end{lemma}
\begin{proof}
Clearly, $J_{n+2}\supseteq\{J_{n+1},J_1\}+\{J_{n},J_2\}.$ The converse is proved by induction on degree $n.$ It is enough to prove the statement for monomials. The base of induction when $n=1$ is evident. Suppose that our statement is true for fewer than $n>1.$  Let $f=\{f_1,f_2\}\in J_{n+2}$, where $f_1\in J_k,f_2\in J_l$ and $k+l=n+2.$ Then by hypothesis of induction and weak right-commutativity we may assume that $f_1$ and $f_2$ have the following form
\begin{center}
$f_1=\{\{\cdots\{\{\{\cdots\{\{a_1,a_2\},\{a_3,a_4\}\},\cdots\},\{a_{p-1},a_p\}\},a_{p+1}\}\cdots\},a_k\},$
\end{center}
\begin{center}
$f_2=\{\{\cdots\{\{\{\cdots\{\{b_1,b_2\},\{b_3,b_4\}\},\cdots\},\{b_{q-1},b_q\}\},b_{q+1}\}\cdots\},b_l\}$\end{center}
 where $p,q\geq2.$ 
$$\{f_1,f_2\}=$$
$$\{\{\{\cdots\{\{\{\cdots\{\{a_1,a_2\},\{a_3,a_4\}\},\cdots\},\{a_{p-1},a_p\}\},a_{p+1}\}\cdots\},a_k\},f_2\}=$$
(by weak right-commutativity)
$$\{\{\cdots\{\{\{\cdots\{\{\{a_1,a_2\},f_2\},\{a_3,a_4\}\},\cdots\},\{a_{p-1},a_p\}\},a_{p+1}\}\cdots\},a_k\}=$$
(by commutativity)
$$\{\{\cdots\{\{\{\cdots\{\{f_2,\{a_1,a_2\}\},\{a_3,a_4\}\},\cdots\},\{a_{p-1},a_p\}\},a_{p+1}\}\cdots\},a_k\}$$
$$\in \{J_{n+1},J_1\}+\{J_{n},J_2\}.$$  
\end{proof}

\begin{lemma}\label{Jordan-symmetric}
Let $f\in J_n.$ Then $f$ is symmetric.  
\end{lemma}
\begin{proof}
Generators of $J_n$ are symmetric. We use induction on degree $n.$ Base of induction is $\{a,b\}=(ab)^{(+)}.$ Assume that the statement is true for elements in degree fewer than $n.$ We have
$$\{\ytableausetup
{mathmode, boxsize=1.2 em}
\begin{ytableau}
a_1 & b_1& b_2 & \dots
&  b_l \\
a_2 \\
\vdots \\
a_k\\
\end{ytableau}^{(+)},\ytableausetup
{mathmode, boxsize=1.2 em}
\begin{ytableau}
c \\
\end{ytableau}\}=\ytableausetup
{mathmode, boxsize=1.2 em}
\begin{ytableau}
a_1 & b_1& b_2 & \dots
&  b_l & c\\
a_2 \\
\vdots \\
a_k\\
\end{ytableau}^{(+)}+\ytableausetup
{mathmode, boxsize=1.2 em}
\begin{ytableau}
a_1 & b_1& b_2 & \dots
&  b_l \\
a_2 \\
\vdots \\
a_k\\
c\\
\end{ytableau}^{(+)},$$ and 
$$\{\ytableausetup
{mathmode, boxsize=1.2 em}
\begin{ytableau}
a_1 & b_1& b_2 & \dots
&  b_l \\
a_2 \\
\vdots \\
a_k\\
\end{ytableau}^{(+)},\ytableausetup
{mathmode, boxsize=1.2 em}
\begin{ytableau}
c & d \\
\end{ytableau}^{(+)}\}=2\,\ytableausetup
{mathmode, boxsize=1.2 em}
\begin{ytableau}
a_1 & b_1& b_2 & \dots
&  b_l & c\\
a_2 \\
\vdots \\
a_k\\
d\\
\end{ytableau}^{(+)}+2\,\ytableausetup
{mathmode, boxsize=1.2 em}
\begin{ytableau}
a_1 & b_1& b_2 & \dots
&  b_l &d\\
a_2 \\
\vdots \\
a_k\\
c\\
\end{ytableau}^{(+)}.$$
By these relations and by Lemma \ref{Filtr-plus}, we see that any Jordan element of degree $n$ is symmetric.\end{proof}
\begin{lemma}\label{symmetric-Jordan}
Let $f$ be a symmetric element of $Bicom(X).$ Then $f$ is a Jordan element.  
\end{lemma} \begin{proof}
Since any symmetric element of $Bicom(X)$ is a linear combination of elements of the form $e^{(+)}_i,$ where $e_i$ is a base element of $Bicom(X),$ it is sufficient to show that $e^{(+)}_i$ is Jordan. We prove it by induction on degree $n.$ The base of induction is given by 
$$((ab)c)^{+}=\frac{1}{2}\{\{a,b\},c\}+\frac{1}{2}\{\{a,c\},b\}-\frac{1}{2}\{\{b,c\},a\}.$$  
Assume that the statement is true for elements in degree fewer than $n.$ We have
$$\ytableausetup
{mathmode, boxsize=1.85 em}
\begin{ytableau}
a_1 & a_2& \dots & a_n
\end{ytableau}^{(+)}=
\frac{1}{2}(\{\ytableausetup
{mathmode, boxsize=1.85 em}
\begin{ytableau}
a_1 & a_2& \dots & a_{n-1}
\end{ytableau}^{(+)},\ytableausetup
{mathmode, boxsize=1.85 em}
\begin{ytableau}
a_n \\
\end{ytableau}\}+$$$$\{\ytableausetup
{mathmode, boxsize=1.85 em}
\begin{ytableau}
a_1 & a_2 & \ldots & a_{n-2} & a_n 
\end{ytableau}^{(+)},\ytableausetup
{mathmode, boxsize=1.85 em}
\begin{ytableau}
a_{n-1} \\
\end{ytableau}\}-\frac{1}{2}\{\ytableausetup
{mathmode, boxsize=1.85 em}
\begin{ytableau}
a_1 & a_2 & \dots & a_{n-2}
\end{ytableau}^{(+)},\ytableausetup
{mathmode, boxsize=1.85 em}
\begin{ytableau}
a_{n-1} & a_n \\
\end{ytableau}^{(+)}\})$$
 and 
$$\ytableausetup
{mathmode, boxsize=1.85 em}
\begin{ytableau}
a_1 & a_2& \dots
&  a_k \\
a_{k+1} \\
\vdots \\
a_n\\
\end{ytableau}^{(+)}=
\{\ytableausetup
{mathmode, boxsize=1.85 em}
\begin{ytableau}
a_1 & a_2& \dots
&  a_k \\
a_{k+1} \\
\vdots \\
a_{n-1}\\
\end{ytableau}^{(+)},\ytableausetup
{mathmode, boxsize=1.85 em}
\begin{ytableau}
a_n \\
\end{ytableau}\}-\ytableausetup
{mathmode, boxsize=1.85 em}
\begin{ytableau}
a_1 & a_2& \dots
&  a_k &a_n \\
a_{k+1} \\
\vdots \\
a_{n-1}\\
\end{ytableau}^{(+)}.$$ 
Therefore, any symmetric element is Jordan. 
\end{proof}

\begin{lemma}\label{irr-iden-plus} Every identity of degree no more than {\rm 4} satisfied by the anti-commutator products in every bicommutative algebra  is a consequence of the identities (\ref{com}), (\ref{Tortken-minus}) and (\ref{weakrcom}). 
\end{lemma}
\begin{proof}
A straightforward calculation shows that there is no identity in degree 3 and it can be easily checked that the identities (\ref{com}), (\ref{Tortken-minus}) and (\ref{weakrcom}) hold in every bicommutative algebra. We show that any other identity in degree 4 follows from  (\ref{com}), (\ref{Tortken-minus}) and (\ref{weakrcom}).
By the identities (\ref{com}), (\ref{Tortken-minus}) and (\ref{weakrcom}), we can have
\begin{equation}\label{rewriting1}
\{\{a,c\},\{b,d\}\}=\{\{a,b\},\{c,d\}\}-\{\{\{a,b\},c\},d\}-
\end{equation}
$$\{\{\{c,d\},a\},b\}+\{\{\{a,c\},b\},d\}+\{\{\{b,d\},a\},c\},$$
\begin{equation}\label{rewriting2}
\{\{a,d\},\{b,c\}\}=\{\{a,b\},\{c,d\}\}-\{\{\{a,b\},c\},d\}-
\end{equation}
$$\{\{\{c,d\},a\},b\}+\{\{\{a,d\},b\},c\}+\{\{\{b,c\},a\},d\}.$$ Therefore, we can claim that the multilinear subspace of $J(\{x,y,z,t\})$ is spanned by the following elements $$\{\{\{x,y\},z\},t\},\{\{\{x,z\},y\},t\},\{\{\{x,t\},y\},z\},\{\{\{y,z\},x\},t\},$$$$\{\{\{y,t\},x\},z\},\{\{\{z,t\},x\},y\},\{\{x,y\},\{z,t\}\}.$$ Hence a new identity in degree 4 has the form 
$$f(x,y,z,t)=\lambda_1\{\{\{x,y\},z\},t\}+\lambda_2\{\{\{x,z\},y\},t\}+\lambda_3\{\{\{x,t\},y\},z\}+$$$$\lambda_4\{\{\{y,z\},x\},t\}+\lambda_5\{\{\{y,t\},x\},z\}+\lambda_6\{\{\{z,t\},x\},y\}+\lambda_7\{\{x,y\},\{z,t\}\}$$ for some $\lambda_1,\ldots,\lambda_7\in{\bf K}.$
We have
$$f(x,y,z,t)=(\lambda_1+\lambda_2+\lambda_3)(((xy)z)t)^{(+)}+(\lambda_2+\lambda_3+\lambda_4+\lambda_5)(y((xz)t))^{(+)}+$$$$(\lambda_1+\lambda_3+\lambda_4+\lambda_6+2\lambda_7)(z((xy)t))^{(+)}+(\lambda_1+\lambda_2+\lambda_5+\lambda_6+2\lambda_7)(t((xy)z))^{(+)}+$$$$(\lambda_3+\lambda_5+\lambda_6)(z(y(xt)))^{(+)}+(\lambda_2+\lambda_4+\lambda_6)(t(y(xz)))^{(+)}+(\lambda_1+\lambda_4+\lambda_5)(t(z(xy)))^{(+)}.$$
Since the elements $$(((xy)z)t)^{(+)}, (y((xz)t))^{(+)}, (z((xy)t))^{(+)}, (t((xy)z))^{(+)},(z(y(xt)))^{(+)},$$$$(t(y(xz)))^{(+)},(t(z(xy)))^{(+)}$$ are linearly independent in $Bicom(\{x,y,z,t\}),$ the condition $f(x,y,z,t)=0$ gives us a system of 7 linear equations with 7 unknowns. We see that the system has rank 7, then $f(x,y,z,t)=0$ if and only if $\lambda_1=\lambda_2=\lambda_3=\lambda_4=\lambda_5=\lambda_6=\lambda_7=0.$  
\end{proof}
\begin{lemma}\label{basis-plus}
$J_n$ is spanned by elements of the form
\begin{equation}\label{basis}
\{\{\cdots\{\{\{\cdots\{\{x_{i_1},x_{i_2}\},\{x_{i_3},x_{i_4}\}\},\cdots\},\{x_{i_{2k-1}},x_{i_{2k}}\}\},x_{i_{2k+1}}\}\cdots\},x_{i_n}\}
\end{equation} where $i_1\leq\ldots\leq i_{2k}$ and $i_{2k+1}\leq\ldots\leq i_n.$ 
In particular, the number of elements of the form {\rm(\ref{basis})} in the multilinear case is equal to $2^{n-1}-1.$    
\end{lemma}
\begin{proof}
By Lemma \ref{Filtr-plus} and the identity (\ref{weakrcom}) we obtain the type of bracketing in (\ref{basis}) for the spanning elements. The identity (\ref{weakrcom}) permits one to order all generators occurring at positions $2k+1,\ldots,n$ in (\ref{basis}). Furthermore, by the relations (\ref{rewriting1}) and (\ref{rewriting2}) we can put the generatorts at positions 1,2,3 and 4 in (\ref{basis}) in non-decreasing order. Then by weak right-commutativity and relations (\ref{rewriting1}),(\ref{rewriting2}), one can order the generators at positions from 1 to $2k$ in (\ref{basis}). 

The number of choosing different variables on the first $2k$ position in (\ref{basis}) is equal to ${n\choose 2k}$ where $k\in\{1,\ldots,\lfloor\frac{n}{2}\rfloor\}.$ Then the number of the spanning elements is equal to $\sum_{k=1}{{n\choose 2k}}=2^{n-1}-1.$
\end{proof}

{\it Proof of Theorem \ref{Jordan identities}}
Lemma \ref{basis-plus} gives the upper bound dimension for the space of multilinear elements of $J_n$ and it is equal to $2^{n-1}-1.$ On the other hand, from Corollary 2.4 in \cite{Dzh-Ism-Tul} follows that if $|X|=n,$ then the multilinear dimension of $Bicom(X)$ is equal to $2^n-2.$ Lemma \ref{Jordan-symmetric} shows that any Jordan element is a linear combination of elements $e^{(+)}$ in $Bicom(X)$ where $e$ is the basis element of space $Bicom(X)$, so the multilinear dimension of subspace of symmetric elements in $Bicom(X)$ is equal to $2^{n-1}-1.$ This gives a lower bound dimension for $J_n.$ Therefore, every identity satisfied by the commutators follows from commutativite, the minus-Tortken and weak right-commutative identities. $\square$

 \section{\label{nn}\ Proof of Theorem \ref{Lie identities}.}
\begin{lemma}\label{Lie expansion}
If $n$ is odd, then
$$[[\cdots[[a_1,a_2],a_3]\cdots],a_n]=$$
$$\ytableausetup
{mathmode, boxsize=1.2 em}
\begin{ytableau}
a_1 & a_2& \dots & a_n
\end{ytableau}^{(+)}-\sum_{i=3}^n\ytableausetup
{mathmode, boxsize=1.2 em}
\begin{ytableau}
a_1 & a_2& \dots & a_n\\
a_i
\end{ytableau}^{(+)}+\ldots+\sum_{i=3}^n\ytableausetup
{mathmode, boxsize=1.2 em}
\begin{ytableau}
a_1 & a_2 & a_i\\
a_3\\
\vdots \\
a_n
\end{ytableau}^{(+)}-\ytableausetup
{mathmode, boxsize=1.2 em}
\begin{ytableau}
a_1 & a_2\\
a_3\\
\vdots \\
a_n\end{ytableau}^{(+)}.$$
If $n$ is even, then
$$[[\cdots[[a_1,a_2],a_3]\cdots],a_n]=$$
$$\ytableausetup
{mathmode, boxsize=1.2 em}
\begin{ytableau}
a_1 & a_2& \dots & a_n
\end{ytableau} ^{(-)}-\sum_{i=3}^n\ytableausetup
{mathmode, boxsize=1.2 em}
\begin{ytableau}
a_1 & a_2& \dots & a_n\\
a_i
\end{ytableau} ^{(-)}+\ldots-\sum_{i=3}^n\ytableausetup
{mathmode, boxsize=1.2 em}
\begin{ytableau}
a_1 & a_2 & a_i\\
a_3\\
\vdots \\
a_n
\end{ytableau} ^{(-)}+\ytableausetup
{mathmode, boxsize=1.2 em}
\begin{ytableau}
a_1 & a_2\\
a_3\\
\vdots \\
a_n\end{ytableau} ^{(-)}.$$
\end{lemma}
\begin{proof}
It is easy to see that   
$$[\,\ytableausetup
{mathmode, boxsize=1.2 em}
\begin{ytableau}
a_1 & b_1& b_2 & \dots
&  b_l \\
a_2 \\
\vdots \\
a_k\\
\end{ytableau}^{(+)},\ytableausetup
{mathmode, boxsize=1.2 em}
\begin{ytableau}
c \\
\end{ytableau}\,]=\ytableausetup
{mathmode, boxsize=1.2 em}
\begin{ytableau}
a_1 & b_1& b_2 & \dots
&  b_l & c\\
a_2 \\
\vdots \\
a_k\\
\end{ytableau} ^{(-)}-\ytableausetup
{mathmode, boxsize=1.2 em}
\begin{ytableau}
a_1 & b_1& b_2 & \dots
&  b_l \\
a_2 \\
\vdots \\
a_k\\
c\\
\end{ytableau} ^{(-)}$$
and
$$[\,\ytableausetup
{mathmode, boxsize=1.2 em}
\begin{ytableau}
a_1 & b_1& b_2 & \dots
&  b_l \\
a_2 \\
\vdots \\
a_k\\
\end{ytableau}^{(-)},\ytableausetup
{mathmode, boxsize=1.2 em}
\begin{ytableau}
c \\
\end{ytableau}\,]=\ytableausetup
{mathmode, boxsize=1.2 em}
\begin{ytableau}
a_1 & b_1& b_2 & \dots
&  b_l & c\\
a_2 \\
\vdots \\
a_k\\
\end{ytableau} ^{(+)}+\ytableausetup
{mathmode, boxsize=1.2 em}
\begin{ytableau}
a_1 & b_1& b_2 & \dots
&  b_l \\
a_2 \\
\vdots \\
a_k\\
c\\
\end{ytableau} ^{(+)}$$
where $k+l=n-1.$ Induction on $n$ based on these relations ends the proof.
\end{proof}
\begin{lemma}\label{irr-iden-minus}
Every identity of degree no more than {\rm 4} satisfied by the commutator products in every bicommutative algebra is a consequence of anti-commutative, the Jacobi and the metabelian identities. 
\end{lemma}
\begin{proof}
In \cite{Burde} and \cite{Dzh-Tul} it was shown that any bicommutative algebra under Lie product becomes a metabelian Lie. A straightforward calculation shows that there is no identity in degree 3. We show that any identity in degree 4 follows from  (\ref{anti-com}), (\ref{jacobi}) and (\ref{metabelian}). If there is a new identity in degree 4 then according to the identities (\ref{anti-com}), (\ref{jacobi}) and (\ref{metabelian}), its polynomial $f(x,y,z,t)$ has the form $$f(x,y,z,t)=\lambda_1[[[x,y],z],t]+\lambda_2[[[x,z],y],t]+\lambda_3[[[x,t],y],z],$$ for some $\lambda_1,\lambda_2,\lambda_3\in{\bf K}.$
We have
$$f(x,y,z,t)=$$$$(\lambda_1+\lambda_2+\lambda_3)(((xy)z)t)^{(-)}-(\lambda_1+\lambda_2)(t((xy)x))^{(-)}-(\lambda_1+\lambda_3)(z((xy)t))^{(-)}-$$$$(\lambda_2+\lambda_3)(y((xz)t))^{(-)}+\lambda_1(t(z(xy)))^{(-)}+\lambda_2(t(y(xz)))^{(-)}+\lambda_3(z(y(xt)))^{(-)}.$$

Since the elements $$(((xy)z)t)^{(+)}, (y((xz)t))^{(+)}, (z((xy)t))^{(+)}, (t((xy)z))^{(+)},(z(y(xt)))^{(+)},$$$$(t(y(xz)))^{(+)},(t(z(xy)))^{(+)}$$ are linearly independent in $Bicom(\{x,y,z,t\}),$ then $f(x,y,z,t)=0$ if and only if $\lambda_1=\lambda_2=\lambda_3=0.$ 
\end{proof} 
A base of a free metabelian Lie algebra is given in \cite{Bahturin}. The set of elements  $[[\cdots[[x_1,x_{i_1}]x_{i_2}],\cdots],x_{i_{n-1}}]$, where $1<i_1$ and  $2\leq i_2\leq\ldots\leq i_{n-1}\leq n$ form a base of the free metabelian Lie algebra of degree $n.$

{\it Proof of Theorem \ref{Lie identities}.}
Suppose that $f(x_1,\ldots,x_n)=0$ is a new identity in degree $n$ satisfied by the commutator product in all bicommutative algebras. We can assume that the polynomial $f(x_1,\ldots,x_n)$ has the following form 
$$f(x_1,\ldots,x_n)=\sum_{i=2}^{n}\lambda_i[[\cdots[[x_1,x_{i}]x_2],\cdots],x_n]$$
for some $\lambda_2,\ldots,\lambda_n\in{\bf K},$ where $1<i_1$ and  $2\leq i_2\leq\ldots\leq i_{n-1}\leq n,$  By Lemma \ref{Lie expansion} we have 
 $$f(x_1,\ldots,x_n)=$$$$(\sum_{i=2}^{n}\lambda_i)\ytableausetup
{mathmode, boxsize=1.8 em}
\begin{ytableau}
x_1 & x_2& \dots & x_n
\end{ytableau}^{(\pm)}\mp\lambda_2\ytableausetup
{mathmode, boxsize=2 em}
\begin{ytableau}
x_1 & x_2\\
x_3\\
\vdots \\
x_n\end{ytableau}^{(\pm)}\mp\ldots\mp\lambda_n\ytableausetup
{mathmode, boxsize=2 em}
\begin{ytableau}
x_1 & x_n\\
x_2\\
\vdots \\
x_{n-1}\end{ytableau}^{(\pm)}+g$$
where $g$ is the sum of base elements which has no form of Young diagrams $(n)$ and $(2,1^{n-2}),$ $f(x_1,\ldots,x
_n)$ takes above signs if $n$ is odd, bottom signs if $n$ is even. Then we see that $f(x_1,\ldots,x_n)=0$ if and only if $\lambda_2=\ldots=\lambda_n=0.$ $\square$
 \section{\label{nn}\ Proof of Theorem \ref{Lie elements}.}
Let $f\in Bicom(X)$ and suppose that $D(head(f))=f.$ Since the image of Dynkin map $D$ is the subspace of $L(X),$ $f$ is a Lie element. 

Now, suppose that $f\in L(X).$ Let $$f=\sum_{i=2}^{n}\lambda_i[[\cdots[[x_1,x_{i}]x_2],\cdots],x_n]$$
for some $\lambda_2,\ldots,\lambda_n\in K.$ Then by Lemma \ref{Lie expansion}
$$head(f)=\mp\lambda_2\ytableausetup
{mathmode, boxsize=1.8 em}
\begin{ytableau}
x_1 & x_2\\
x_3\\
\vdots \\
x_n\end{ytableau}^{(\pm)}\mp\ldots\mp\lambda_n\ytableausetup
{mathmode, boxsize=1.8 em}
\begin{ytableau}
x_1 & x_n\\
x_2\\
\vdots \\
x_{n-1}\end{ytableau}^{(\pm)}$$
If $n$ is odd, we note that
$$D(-\lambda_2\ytableausetup
{mathmode, boxsize=1.8 em}
\begin{ytableau}
x_1 & x_2\\
x_3\\
\vdots \\
x_n\end{ytableau}^{(+)})=-\lambda_2D(\ytableausetup
{mathmode, boxsize=1.8 em}
\begin{ytableau}
x_1 & x_2\\
x_3\\
\vdots \\
x_n\end{ytableau}+\ytableausetup
{mathmode, boxsize=1.8 em}
\begin{ytableau}
x_2 & x_1& \dots & x_n
\end{ytableau})
$$$$-\frac{\lambda_2}{2}([x_n,[\cdots[x_3,[x_1,x_2]]\cdots]]+[[\cdots[[x_2,x_1],x_3]\cdots],x_n])=$$
(by anti-commutativity) 
$$\lambda_2[[\cdots[[x_1,x_2],x_3]\cdots],x_n].$$
If $n$ is even, we note that
$$D(\lambda_2\ytableausetup
{mathmode, boxsize=1.8 em}
\begin{ytableau}
x_1 & x_2\\
x_3\\
\vdots \\
x_n\end{ytableau}^{(-)})=\lambda_2D(\ytableausetup
{mathmode, boxsize=1.8 em}
\begin{ytableau}
x_1 & x_2\\
x_3\\
\vdots \\
x_n\end{ytableau}-\ytableausetup
{mathmode, boxsize=1.8 em}
\begin{ytableau}
x_2 & x_1& \dots & x_n
\end{ytableau})
$$$$\frac{\lambda_2}{2}([x_n,[\cdots[x_3,[x_1,x_2]]\cdots]]-[[\cdots[[x_2,x_1],x_3]\cdots],x_n])=$$
(by anti-commutativity) 
$$\lambda_2[[\cdots[[x_1,x_2],x_3]\cdots],x_n].$$
Then apply Dynkin map $D$ to other basis elements in $head(f)$ and we obtain $D(head(f))=f.$ 
 \section{\label{nn}\ Proof of Corollary \ref{relation}.} 
It follows from Theorem \ref{Jordan element} and Lemma \ref{Lie expansion}.  
\section{\label{nn}\ Independence of identities in degree 4.}
In this section we show independence of the minus-Tortken and weak right-commutativity identities in a free bicommutative algebra. Let $A=\{e_1,e_2,e_3,e_4\}$ be a four dimensional algebra with the multiplication table
$$e^2_1=e_1, \quad  e_1e_2=e_2e_1=\frac{1}{2}e_2 $$
	and zero products are omitted. This algebra is a Jordan algebra \cite{Martin}. A straightforward calculation shows that the algebra $A$ satisfies the identity (\ref{Tortken-minus}). But for substition $a=e_1,b=e_1,c=e_1$ and $d=e_2$, $A$ does not satisfy the weak right-commutative identity, namely $$((e_1e_1)e_1)e_2-((e_1e_1)e_2)e_1=\frac{e_2}{4}\neq0.$$ So, the weak right-commutativity is not a consequence of the minus-Tortken identity in every commutative algebra.
	  
It is easy to see that a set of the mutlinear bases elements in degree 4 of the free commutative algebra with identity (\ref{weakrcom}) is the following $$\{((ab)c)d,((ac)b)d,((ad)b)c,((bc)a)d,((bd)a)c,((cd)a)b,$$ 
$$(ab)(cd),(ac)(bd),(ad)(bc)\}.$$ We express each summand in the identity (\ref{Tortken-minus}) in terms of these base elements and one may see that a free commutative algebra with identity (\ref{weakrcom}) does not satisfy the minus-Tortken identity. So, the minus-Tortken identity is not a consequence of weak right-commutativity in every commutative algebra. Therefore, they are independent identities in a free bicommutative algebra.


\begin{thebibliography}{25}
\bibitem{Bahturin}Yu. A. Bahturin, {\em Identical relations in Lie algebras}, VNU Science Press, Utrecht, 1987.Translated from the Russian by Bahturin.
\bibitem{Ber-Lod} N. Bergeron, J.-L. Loday {\em  The symmetric operation in a free pre-Lie algebra is magmatic},
Proc. Amer. Math. Soc. 139 (2011) no.5, 1585-1597.
\bibitem{Bergman} G.M. Bergman {\em  The diamond lemma for the ring theory}, Adv. in Math, 29 (1978),no 2., 178-218. 
\bibitem{Bokut1} L.A. Bokut {\em  Imbeddings into simple associative algebras},
Algebra i Logika, 15 (1976), no.2, 117-142. 
\bibitem{Bokut2} L.A. Bokut, K.P, Schum {\em Gr$\ddot{o}$bner and Gr$\ddot{o}$bner-Shirshov bases in algebra: an elementary approach.}, Southeast Asian Bull. Math, 29 (2005), no.2, 227-252.
\bibitem{Bokut3} L.A. Bokut, Y. Chen, Y.Li  {\em Gr$\ddot{o}$bner-Shirshov bases for Vinberg-Koszul-Gerstenhaber right-symmetric algebras}, Fundam.Prikl. Mat. 14 (2008) no.8, 55-67. 
\bibitem{Burde} Dietrich Burde, Karel Dekimpe,and S Deschamps {\em  LR-algebras},
New Developments in Lie Theory and Geometry, Amer. Math. Soc., Providence, RI, 
Contemporary Mathematics 491 (2009), 125-140.
\bibitem{Cohn}
Cohn P.M. {\em On homomorphic images of special Jordan algebras}
Canadian J. Math. V. 6(1954), 253-264. 
\bibitem{Drensky-Zhakh} V. Drensky, B. Zhakhayev {\em Noetherianity and Specht problem for varieties of bicommutative algebras},  arXiv:1706.02529.
\bibitem{Drensky} V. Drensky{\em Varieties of bicommutative algebras},   arXiv:1706.04279.
\bibitem{Dynkin} Dynkin, E. B., {\em Computation of the coeffients in the Campbell-Hausdroff formula,} Doklady Akad. Nauk SSSR 57 (1947), 323-326.
\bibitem{Dzhum-Lof} A.S. Dzhumadil'daev, C. L\"ofwall {\em  Trees, free
right-symmetric algebras, free Novikov algebras and identities},
Homology, Homotopy and Appl.,4 (2002), No.2(1), 165-190.
\bibitem{Dzhumadil'daev2} A.S. Dzhumadil'daev {\em Novikov-Jordan algebras} Comm. Algebra. V. 30(11) (2002). 5205-5240.
\bibitem{Dzhumadil'daev3} A.S. Dzhumadil'daev {\em Special identity for Novikov-Jordan algebras} Comm. Algebra. V. 33(5) (2005), 1279-1287.
  \bibitem{Dzh-Tul} A.S. Dzhumadil'daev, K.M. Tulenbaev,  {\em Bi-commutative algebras,} Uspechi Math. Nauk, 58 (2003), No.6, 149-150=engl.transl.Russian Math. Surv., 1196-1197.
\bibitem{Dzhumadil'daev4} A.S. Dzhumadil'daev {\em Zinbiel algebras under $q$-commutators } Journal of Mathematical Sciences. V. 144 (2007), No 2, 3909-3925.  
\bibitem{Dzhumadil'daev5} A.S. Dzhumadil'daev,  {\em q-Leibniz algebras}, Serdica Math.J.34,(2008), 415-440.
\bibitem{Dzhumadil'daev6} A.S. Dzhumadil'daev,  {\em Jordan elements and left-center of a free Leibniz algebra,} Elevtronic Research Announcements In Mathemcatical Sciences Vol. 18,(2011), 1-49.
\bibitem{Dzh-Ism-Tul} A.S. Dzhumadil'daev, N.A. Ismailov, K.M. Tulenbaev {\em  Free bicommutative algebras}, Serdica Math. J., 37 (2011), No. 1, 25-44.
\bibitem{Dzhumadil'daev7} A.S. Dzhumadil'daev {\em  Assosymmetric algebras under Jordan product}, Com. Algebra, http://dx.doi.org/10.1080/00927872.2017.1327054.
\bibitem{Friedrichs}
K.O. Friedrichs {\em Mathematical aspects of the quantum theory of fields.} V, Comm. Pure Appl. Math. 6(1953), 1-72.
\bibitem{Glennie1} C.M. Glennie {\em Some identities valid in special Jordan algebras but not valid in all Jordan algebras}, Pacific J. Math, 16 (1966), 47-59.
\bibitem{Glennie2} C.M. Glennie {\em Identities in Jordan algebras}, Computational Problems in Abstract Algebra (Proc. Conf., Oxford, 1967), Pergamon, Oxford, (1970), 307-313.
\bibitem{Kaygorodov-Volkov} Kaygorodov I.,Volkov Yu., {\em The variety of $2$-dimensional algebras over an algebraically closed field}, arXiv:1701.08233.
\bibitem{Manchon} D. Manchon {\em A short survey on pre-Lie algebras.} Noncommutative Geometry and Physics:Renormalisation, Motives, Index Theory, 89-102, ESI Lect. Math. Phys., Eur. Math. Soc., Z\"urich, (2011).
\bibitem{Markl} M. Markl, {\em Lie elements in pre-Lie algebras, trees and cohomology operations.} J. Lie Theory
17 (2007), no. 2, 241-261.
\bibitem{Martin} M.E.Martin, {\em Four dimensional Jordan algebras}, Int. J. Math. Game Theory Algebra 20 (2013), 41-59.
\bibitem{McCrimmon} K. McCrimmon {\em The role of identities in Jordan algebras}, Resenhas 6 (2004), no.2-3, 265-280.
\bibitem{Robbins}
D.P. Robbins {\em Jordan elements in a free associative algebra, I} J. Algebra 19(1971), 354-378.
\bibitem{Specht} W. Specht, {\em Die lineare Beziehungen zwischen h\"oheren Kommutatoren}, Math. Z. 51(1948), 367-376.
\bibitem{Wever} W. Wever, {\em \"Uber Invarienten in Lieschen Ringe}, Math.Ann 120(1948), 563-580.
\bibitem{ZSSS} K.A. Zhevlakov, A.M. Slinko, I.P. Shestakov, A.I. Shirshov, {\em Rings That Are Nearly Associative.} Moscow: Nauka. (1976) 
\end{thebibliography}
\end{document}